\documentclass[10pt]{amsart}

\usepackage{hyperref}

\newtheorem{theorem}{Theorem}
\newtheorem{lemma}{Lemma}

\newtheorem{prop}{Proposition}

\newtheorem*{hyp}{Hypothesis B}

\theoremstyle{remark}
\newtheorem{remark}{Remark}
\begin{document}

\title{Linear recursive odometers and beta-expansions}

\author[M. R. Iac\`o]{Maria Rita Iac\`o}
\address{Maria Rita Iac\`o \newline
\indent IRIF, CNRS UMR 8243, \newline
\indent Universit\'e Paris Diderot -- Paris 7,\newline
\indent Case 7014, 75205 Paris Cedex 13.}
\email{maria-rita.iaco@liafa.univ-paris-diderot.fr}

\author[W. Steiner]{Wolfgang Steiner}
\address{Wolfgang Steiner \newline
\indent IRIF, CNRS UMR 8243, \newline
\indent Universit\'e Paris Diderot -- Paris 7,\newline
\indent Case 7014, 75205 Paris Cedex 13.}
\email{steiner@liafa.univ-paris-diderot.fr}

\author[R. F. Tichy]{Robert F. Tichy}
\address{Robert F. Tichy \newline
\indent Graz University of Technology, \newline
\indent Institute for Analysis and Number Theory,\newline
\indent  Steyrergasse 30, 8010 Graz, Austria.}
\email{tichy@tugraz.at}

\thanks{The authors are participants in the ANR/FWF project FAN \lq\lq Fractals and Numeration\rq\rq\
(ANR-12-IS01-0002, FWF grant I1136).}

\begin{abstract}
The aim of this paper is to study the connection between different properties related to $\beta$-expansions. In particular, the relation between two conditions, both ensuring pure discrete spectrum of the odometer, is analysed. The first one is the so-called Hypothesis B for the $G$-odometers and the second one is denoted by (QM) and it has been introduced in the framework of tilings associated to Pisot $\beta$-numerations.

\end{abstract}

\dedicatory{Dedicated to the memory of Pierre Liardet}
\maketitle

\section{Introduction}
In the early 1990's Pierre Liardet visited Graz several times and
Pierre, Peter Grabner and the third author started an intensive cooperation on dynamic properties of digital expansions. Their main results from that period where published in~\cite{glt}. In this work, the basic theory
of odometers was developed. In the subsequent years, these aspects of arithmetic dynamics were extended by various authors. Several PhD students of Pierre worked in this field, in particular we want to mention Guy Barat who received his PhD in Marseilles 1995 and his habilitation at Graz University of Technology 2006.

A~special focus lies on arithmetic conditions which guarantee purely discrete spectrum of the odometer.
In~\cite{HIT}, and before in~\cite{glt}, the authors posed the question whether Hypothesis~B, introduced in~\cite{glt}, and the finiteness property~\eqref{F}, introduced in~\cite{frougny}, are equivalent.
Hypothesis~B is a condition on the carries of the digits in the expansion of positive integers in a base system defined by a linear recurrence. 

Let $(G_k)_{k\ge0}$ be an increasing sequence of positive integers, with initial value $G_0=1$.
Then every positive integer can be expanded as
\begin{equation*}
  n=\sum_{k=0}^{\infty}\varepsilon_k(n) G_k\ ,
\end{equation*}
where $\varepsilon_k(n) \in \{0, \ldots, \lceil G_{k+1}/G_k \rceil -1 \}$ and $\lceil x \rceil$ denotes the smallest integer not less than $x \in \mathbb{R}$. This expansion (called $G$-expansion) is uniquely determined and finite, provided that for every $K$,
\begin{equation}\label{eq1}
\sum_{k=0}^{K-1}\varepsilon_k(n) G_k < G_K .
\end{equation}
For short we write $\varepsilon_k$ for the $k$-th digit of the $G$-expansion; $G = (G_k)_{k \geq 0}$ is called numeration system and the digits $\varepsilon_k$ can be computed by the greedy algorithm.

We denote by $\mathcal{K}_G$ the subset of sequences that satisfy (\ref{eq1}) and we call its elements $G$-admissible. 
In order to extend the addition-by-one map $\tau$ defined on $\mathbb{N}$ to $\mathcal{K}_G$ the following subset of $\mathcal{K}_G$ is introduced:
\begin{equation}
 \mathcal{K}_G^0 = \bigg\{ x\in \mathcal{K}_G\ : \exists M_x, \forall j\geq M_x, \quad \sum_{k=0}^{j}\varepsilon_k G_k\ < G_{j+1}-1 \bigg\}\ .
\end{equation}
Put $x(j)=\sum_{k=0}^{j}\varepsilon_k G_k$, and set
\begin{equation}\label{eq2}
 \tau(x)=(\varepsilon_0(x(j)+1)\dots \varepsilon_j(x(j)+1))\varepsilon_{j+1}(x)\varepsilon_{j+2}(x)\dots \ ,
\end{equation}
for every $x\in \mathcal{K}_G^0$ and $j\geq M_x$. This definition does not depend on the choice of $j\geq M_x$. 
We extend the definition of $\tau$ to sequences $x$ in $\mathcal{K}_G\setminus \mathcal{K}_G^0$ by $\tau(x)=0=(0^{\infty})$; in this way, the transformation $\tau$ is defined on $\mathcal{K}_G$ and it is called $G$-odometer. 

As in~\cite{glt}, we consider sequences $(G_k)_{k\ge0}$ associated to real numbers $\beta > 1$, defined by
\begin{equation} \label{e:Gk}
G_k = \sum_{j=1}^k a_j G_{k-j} + 1, 
\end{equation}
where $a_1 a_2 \cdots$ is the quasi-greedy $\beta$-expansion of~$1$, i.e., the smallest sequence (w.r.t.\ the lexicographical order) containing infinitely many non-zero digits and satisfying $1 = \sum_{j=1}^\infty a_j \beta^{-j}$.
By~\cite{glt}, the $G$-odometer $(\mathcal{K}_G, \tau)$ is continuous if and only if $a_1 a_2 \cdots$ is purely periodic.
Note that, when $a_1 a_2 \cdots$ has period length~$d$, we have
\[
G_k = \sum_{j=1}^d a_j G_{k-j} + G_{k-d} \quad \mbox{for all}\ k \ge d.
\] 
For purely periodic $a_1 a_2 \cdots$, it was shown in~\cite{glt} also that $(\mathcal{K}_G, \tau)$ is uniquely ergodic, an explicit formula for the unique invariant measure $\mu$ defined on $\mathcal{K}_G$ is provided.
For more general $G$-expansions, unique ergodicity follows from the work of Barat and Grabner~\cite{Barat-Grabner16}.
The following condition is used to prove that $\mathcal{K}_G$ has purely discrete spectrum.


\begin{hyp}
There exists an integer $b \ge 0$ such that for all $k$ and 
\begin{equation*}
N = \sum_{i=0}^{k-1} \epsilon_i(N) G_i + \sum_{j=k+b}^\infty \epsilon_j(N) G_j,
\end{equation*}
the addition of $G_m$ to $N$, where $m \geq k + b$, does not change the first $k$ digits in the greedy representation, i.e.,
\begin{equation*}
N + G_m = \sum_{i=0}^{k-1} \epsilon_i(N) G_i + \sum_{j=k}^\infty \epsilon_j(N + G_m) G_j.
\end{equation*}
\end{hyp}


The finiteness property (F) is instead defined in the framework of $\beta$-expansions. Let $\beta >1$ be a fixed real number. A~$\beta$-expansion of a real number $x \in [0,1)$ is a representation of the form
\begin{equation*}
x=\sum_{i=1}^\infty \epsilon_i\beta^{-i}\ ,
\end{equation*}
where $\epsilon_i\in\{0,1,\dots, \lceil \beta\rceil -1\}$ and $\lceil x\rceil$ denotes the smallest integer not less than~$x$. 
Beta-expansions were introduced by R\'enyi~\cite{Renyi} and generalize standard representations in an integral base. 
These expansions can be obtained via the iteration of the so-called $\beta$-transformation~$T_\beta$ defined by 
\begin{equation*}
T_\beta\colon [0,1)\rightarrow [0,1)\ , \quad x\mapsto \beta x-\lfloor\beta x \rfloor\ ,
\end{equation*}
where $\lfloor x \rfloor$ is the largest integer not exceeding~$x$. 
Taking, at each iteration of~$T_\beta$, $\epsilon_i=\lfloor \beta T_\beta^{i-1}(x) \rfloor$, we obtain the following greedy expansion of $x$
\begin{equation*}
x=\sum_{k=1}^\infty\epsilon_k\beta^{-k}=0.\epsilon_1\epsilon_2\epsilon_3\dots\ .
\end{equation*}
To obtain the quasi-greedy $\beta$-expansions, one can use the transformation 
\[
\tilde{T}_\beta:\ (0,1]\to (0,1], \quad x \mapsto \beta x - \lceil \beta x \rceil + 1,
\]
which differs from~$T_\beta$ only at the points of discontinuity.
Then the quasi-greedy $\beta$-expansion of~$1$ is given by $a_j = \lceil \beta \tilde{T}_\beta^{j-1}(1) \rceil-1$.
Let
\[
V_\beta = \big\{\tilde{T}_\beta^k(1): k\geq 0\big\}.
\]
If $V_\beta$ is finite, then $\beta$ is called a Parry number. 
As for $G$-adic expansions, not all strings of digits in $\{0,1,\dots, \lceil \beta\rceil -1\}$ are admissible. Parry \cite{Parry} observed that 
a sequence $\epsilon_1 \epsilon_2 \cdots$ is admissible if and only if 
\begin{equation}\label{admissibility}
\epsilon_{j}\epsilon_{j+1}\dots < a_1a_2\dots\qquad {\rm for\ all\ } j\geq 1\ .
\end{equation}
A sequence is the $\beta$-expansion of some $x \in [0,1)$ if and only if it is admissible. 

A~significant question in this setting is for which $\beta$ is the expansion in base $\beta$ of $x$ finite, i.e., it is important to provide a description of the set 
\[
\mathrm{Fin}(\beta) = \{x\in[0,1):\, \exists k\geq 0,\ T_\beta^k(x)=0\}.
\] 
Note that many authors rather consider $x \in [0,\infty)$ in the definition of $\mathrm{Fin}(\beta)$, with the condition that $T_\beta^k(\beta^{-n} x)=0$ for $x \in [0,\beta^n)$.
A~number $\beta$ is said to have the \emph{finiteness property} if 
\begin{equation}\tag{F} \label{F}
\mathrm{Fin}(\beta) = \mathbb{Z}[\beta^{-1}] \cap [0,1)
\end{equation}
holds.
This property was introduced by Frougny and Solomyak~\cite{frougny}, and they proved \cite[Lemma~1]{frougny} that if \eqref{F} holds, then $\beta$ is a Pisot number. 
An algebraic integer $\beta >1$ is called a Pisot number if all its Galois conjugates have modulus less than~1. However, there exist Pisot numbers that do not fullfill \eqref{F}, such as all numbers with non purely periodic quasi-greedy $\beta$-expansion of~$1$. 
In~\cite{frougny}, it is also shown that if $a_1\geq a_2\geq a_3 \geq \cdots$, then 
\begin{equation} \tag{PF} \label{PF}
\mathbb{Z}_+[\beta^{-1}] \cap [0,1) \subseteq \mathrm{Fin}(\beta)\ ,
\end{equation}
where $\mathbb{Z}_+=\mathbb{Z}\cap [0,\infty)$. 
This condition is usually referred to as the \emph{positive finiteness} condition~\eqref{PF}, and it is equivalent to say that $\bigcup_{n\ge0} \beta^n \mathrm{Fin}(\beta)$ is closed under addition. Akiyama \cite[Theorem 1]{akiyama2006} proved that if $\beta > 1$ is a real number satisfying~\eqref{PF}, then $\beta$ satisfies \eqref{F} or $a_1 \ge a_2 \ge \cdots$.
This result will be used in the proof of Lemma~\ref{l:PF+QM-B}.

In the present paper we show that Property~\eqref{F} does not imply Hypothesis~B. 
We show that we also need the so-called \emph{quotient mapping condition}
\begin{equation} \tag{QM} \label{QM}
\mathrm{rank}(\langle V_\beta - V_\beta\rangle_\mathbb{Z}) = \deg(\beta) - 1,
\end{equation}
where $\langle V_\beta - V_\beta\rangle_\mathbb{Z}$ denotes the $\mathbb{Z}$-module spanned by differences of elements of~$V_\beta$. 
This condition was introduced by Siegel and Thuswaldner~\cite{ST} in the framework of tilings associated to Pisot substitutions, and for $\beta$-expansions in the present form in~\cite{MS}.
Moreover, if we allow sequences $a_1 a_2 \cdots$ that are not purely periodic, then Hypothesis~B does not imply~\eqref{F} but only~\eqref{PF}.

\begin{theorem} \label{t:3}
Let $\beta > 1$.
Hypothesis~B holds for the sequence $(G_k)_{k\ge0}$ associated to~$\beta$ if and only if conditions \eqref{PF} and \eqref{QM} hold.
\end{theorem}

Since $\langle V_\beta\rangle_\mathbb{Z} = \mathbb{Z}[\beta]$, condition~\eqref{QM} holds when $\# V_\beta = \deg(\beta)$, i.e., when $G$ satisfies a linear recurrence with the minimal polynomial of~$\beta$ as characteristic polynomial. 
A~class of non-trivial examples of bases satisfying~\eqref{QM} was given in \cite{ST, MS} by $\beta^3 = t \beta^2 - \beta + 1$, $t \ge 2$; in this case, we have $\# V_\beta = \deg(\beta) + 1$.
The following theorem gives a characterization of~\eqref{QM} for $\beta > 1$ satisfying $\# V_\beta = \deg(\beta) + 1$.

\begin{theorem} \label{t:2}
Let $\beta > 1$ be such that $\# V_\beta = \deg(\beta) + 1$, with $\tilde{T}_\beta^{\deg(\beta)+1}(1) = \tilde{T}_\beta^k(1)$, $0 \le k \le \deg(\beta)$, i.e., $a_1 a_2 \cdots = a_1 \cdots a_k\, (a_{k+1} \cdots a_{\deg(\beta)+1})^\infty$. 
Then $\beta$ satisfies \eqref{QM} if and only if $\deg(\beta) - k$ is even. 
\end{theorem}

In particular, when $\beta$ is a simple Parry number with $\# V_\beta = \deg(\beta) + 1$, we have $k = 0$ and thus \eqref{QM} holds if and only if $\deg(\beta)$ is odd, e.g., for $\beta^3 = 3\beta^2 - 2\beta + 2$ or $\beta^3 = 3\beta^2 - \beta + 1$. 
Of course, it would be interesting to know what happens if we drop the condition $\# V_\beta = n = \deg(\beta) + 1$ and if there still exist numbers $\beta$ for which (QM) holds.

We conclude with a theorem showing that the odometer has purely discrete spectrum when $\beta$ is a Pisot number satisfying~\eqref{QM}, even when \eqref{PF} does not hold.
Its proof is based on recent results by Barge~\cite{Barge}.

\begin{theorem} \label{t:1}
Let $\beta$ be a Pisot number satisfying~\eqref{QM}. 
Then the odometer $(\mathcal{K}_G,\tau_G)$ associated to~$\beta$ has purely discrete spectrum (with respect to the unique invariant measure). 
\end{theorem}

We do not know whether \eqref{QM} is a necessary condition for purely discrete spectrum.

\section{Quotient mapping condition} \label{sec:quot-mapp-cond}
We first explain the relation of the condition~\eqref{QM} above to the quotient mapping condition defined in~\cite{ST} for subsitutions.
Let $\sigma$ be a primitive substitution on a finite alphabet~$A$ and $M_\sigma = (|\sigma(j)|_i)_{i,j\in A}$ its incidence matrix, where $|\sigma(j)|_i$ denotes the number of occurrences of the letter~$i$ in~$\sigma(j)$.
Let $\mathbf{v} = (v_1,v_2,\ldots,v_n)$ be a left eigenvector of~$M_\sigma$ to the dominant eigenvalue $\beta > 1$, with $v_i \in \mathbb{Q}(\beta)$, and 
\[
L_\sigma = \langle v_i - v_j:\, i,j \in A\rangle_\mathbb{Z}
\]
be the $\mathbb{Z}$-module generated by the differences of coordinates of~$\mathbf{v}$. 
Note that $\mathbf{v}$ and $L_\sigma$ are only defined up to a constant factor, which plays no role in the following. 
The substitution $\sigma$ satisfies the quotient mapping condition if
\[
\mathrm{rank}(L_\sigma) = \deg(\beta) - 1.
\]
This definition is equivalent to Definition~3.13 in \cite{ST}.

Let now $\beta > 1$ be a Parry number, with $\# V_\beta = n$, and $a_1 \cdots a_k (a_{k+1} \cdots a_n)^\omega$ its quasi-greedy $\beta$-expansion of~$1$. 
Then the \emph{$\beta$-substitution}~$\sigma_\beta$ is defined on $A = \{1,2,\ldots,n\}$ by
\begin{align*}
\sigma_\beta: \quad & i \mapsto \underbrace{1\,1\,\cdots\,1}_ {a_i\,\text{times}}\, (i+1) \quad \mbox{if}\ 1 \le i < n, \\
& n \mapsto \underbrace{1\,1\,\cdots\,1}_ {a_n\,\text{times}}\, (k+1).
\end{align*}
Then $(1,\tilde{T}_\beta(1),\ldots,\tilde{T}_\beta^{n-1}(1))$ is a left eigenvector of~$\sigma_\beta$, and $L_{\sigma_\beta} = \langle V_\beta - V_\beta\rangle_\mathbb{Z}$, thus \eqref{QM} holds if and only if the $\beta$-substitution satisfies the quotient mapping condition of~\cite{ST}. 

For an algebraic number $\beta$ with $r$ real and $s$ complex conjugates $\beta_1, \ldots, \beta_r$, $\beta_{r+1},\ldots, \beta_{r+s}$, set
\[
\delta_\infty:\, \mathbb{Q}(\beta) \to \mathbb{R}^r \times \mathbb{C}^s, \quad x \mapsto (x^{(1)}, \ldots, x^{(r)}, x^{(r+1)}, \ldots, x^{(r+s)}),
\]
where $x \mapsto x^{(i)}$ is the Galois embedding $\mathbb{Q}(\beta) \to \mathbb{R}$ or $\mathbb{C}$ that maps $\beta$ to~$\beta^{(i)}$. 

\begin{prop} \label{p:subspace}
A primitive substitution $\sigma$ satisfies the quotient mapping condition if and only if there exists $c \in \mathbb{Q}(\beta)$ such that, for the scalar product, $\delta_\infty(c) \cdot \delta_\infty(v_i) = 1$ for all $i \in A$.
\end{prop}

\begin{proof}
If the quotient mapping condition holds, then there exists $c \in \mathbb{Q}(\beta)$ such that $\delta_\infty(c) \cdot \delta_\infty(x) = 0$ for all $x \in L_\sigma$. 
This implies that $\delta_\infty(c) \cdot \delta_\infty(v_i) = \delta_\infty(c) \cdot \delta_\infty(v_j)$ for all $i, j \in A$, with $q = \delta_\infty(c) \cdot \delta_\infty(v_i) \in \mathbb{Q}$. 
Then we have $\delta_\infty(c/q) \cdot \delta_\infty(v_i) = 1$ for all $i \in A$.

For the other direction, suppose that $\delta_\infty(c) \cdot \delta_\infty(v_i) = 1$ for all $i \in A$, thus $\delta_\infty(c) \cdot \delta_\infty(x) = 0$ for all $x \in L_\sigma$. 
Then $L_\sigma$ has rank at most $\deg(\beta) - 1$.
Since $\langle v_i:\, i \in A\rangle_\mathbb{Z}$ has full rank $\deg(\beta)$, the rank of $L_\sigma$ is also at least $\deg(\beta) - 1$. 
\end{proof}

Note that $\delta_\infty(c) \cdot \delta_\infty(v_i) = 1$ for all $i \in A$ means that the vector $(1,1,\ldots,1) \in \mathbb{Z}^n$ lies in the subspace spanned by the left eigenvectors of~$M_\sigma$ to the eigenvalues that are Galois conjugates of~$\beta$. 

We can now prove the characterization of~\eqref{QM} for $\beta$ with $\# V_\beta = \deg(\beta) + 1$. 

\begin{proof}[Proof of Theorem~\ref{t:2}]
If $\# V_\beta = n = \deg(\beta) + 1$, then the eigenvalues of $M_{\sigma_\beta}$ are the conjugates of~$\beta$ and~$-1$. 
Note that $1$ cannot be an eigenvalue of $M_{\sigma_\beta}$ because the characteristic polynomial of~$M_{\sigma_\beta}$ is 
\begin{equation} \label{e:betapolynomial}
(x^n - a_1 x^{n-1} - a_2 x^{n-2} - \cdots - a_n) - (x^k - a_1 x^{k-1} - a_2 x^{k-2} - \cdots - a_k).
\end{equation}
The right eigenspace to the eigenvalue~$-1$ is spanned by $\mathbf{w} = {}^t(w_1,w_2,\ldots,w_n)$ with 
\[
w_i = \left\{\begin{array}{cl} (-1)^i & \mbox{if}\ k < i \le n,\\[.5ex] (-1)^i\,\big(1-(-1)^{n-k}\big) & \mbox{if}\ 1 \le i \le k. \end{array}\right.
\]
Indeed, we have $w_i + w_{i+1} = 0$ for $1 \le i < k$ and $k < i < n$, and $w_k + w_{k+1} + w_n = 0$.
By Proposition~\ref{p:subspace}, \eqref{QM} is equivalent to the vector $\mathbf{1} = (1,1,\ldots,1)$ lying in the subspace spanned by the left eigenvectors corresponding to the conjugates of~$\beta$. 
This means that $\mathbf{1}$ is orthogonal to~$\mathbf{w}$, i.e., that $n - k$ is even. 
\end{proof}

\section{Equivalence of Hypothesis~B and (PF) \& (QM)}
In this section, let $(G_k)_{k\ge 0}$ be a sequence associated to $\beta > 1$, as defined in~\eqref{e:Gk}.
If $\beta$ is a Parry number, then we can write
\begin{equation} \label{e:GM}
G_k = (1,1,\ldots,1)\, M_{\sigma_\beta}^k\, {}^t(1,0,\ldots,0).
\end{equation}

\begin{lemma} \label{l:QMminimal}
Property \eqref{QM} holds if and only if $(G_k)_{k\ge0}$ satisfies a recurrence with the minimal polynomial of~$\beta$ as characteristic polynomial.
\end{lemma}

\begin{proof}
If \eqref{QM} holds, then, by \eqref{e:GM} and Proposition~\ref{p:subspace}, $G_k$~satisfies a recurrence with the minimal polynomial of~$\beta$ as characteristic polynomial.

If \eqref{QM} does not hold, then $\mathbf{1} = (1,1,\ldots,1)$ does not lie in the subspace spanned by the eigenvectors corresponding to the conjugates of~$\beta$.
Let $\mathbf{1} = \mathbf{b} + \mathbf{c}$ be the decomposition in a vector $\mathbf{b}$ lying in this subspace and $\mathbf{c}$ lying in the complementary invariant subspace. 
By the structure of $M_{\sigma_\beta}$, we obtain that $\mathbf{c}\, M_{\sigma_\beta}^k {}^t(1,0,\ldots,0) \ne 0$ for some $k \ge 0$, thus $G_j$ does not satisfy a recurrence with the minimal polynomial of~$\beta$ as characteristic polynomial.
\end{proof}

\begin{lemma} \label{l:BPF}
Hypothesis~B implies~\eqref{PF}, in particular $\beta$ is a Pisot number. 
\end{lemma}

\begin{proof}
The proof is done by contradiction. Assume that \eqref{PF} does not hold.
Then there is some $y \in \mathbb{Z}_+[\beta^{-1}] \cap [0,1)$ with $y \notin \mathrm{Fin}(\beta)$.
We can choose~$y$ minimal in the sense that $x = y - \beta^{-n} \in \mathbb{Z}_+[\beta^{-1}] \cap \mathrm{Fin}(\beta)$ for some $n > 0$. 
Let $x = \sum_{j=1}^\ell x_j \beta^{-j}$ be the (finite) $\beta$-expansion of~$x$, $y = \sum_{j=1}^\infty y_j \beta^{-j}$ the (infinite) $\beta$-expansion of~$y$. 
Suppose that Hypothesis~B holds for some $b > 0$.
Choose $h > \ell+b$ such that $y_h \ne 0$, and set $N_k = \sum_{j=1}^\ell x_j G_{k-j}$, $N_k'=N_k+G_{k-n}$ for $k\geq \ell$.
We show that $y_j = \epsilon_{k-j}(N_k')$ for all $1 \le j \le h$ and sufficiently large~$k$, contradicting Hypothesis~B since $N_k = \sum_{j=k-\ell}^\infty \epsilon_j(N_k) G_j$ and $\epsilon_{k-h}(N_k + G_{k-n}) = y_h \ne 0 = \epsilon_{k-h}(N_k)$. 

To find the $G$-expansion of~$N_k'$, recall that $G_j = c\,\beta^j+\mathcal{O}(\alpha^j)$ for some constant $c > 0$ and $0 < \alpha <\beta$; see e.g.\ \cite{bertrand, Ito_Takahashi}. 
We have thus $N_k' = c\, y\, \beta^k + \mathcal{O}(\alpha^k) + \mathcal{O}(1)$, and 
\[
N_k' - \sum_{j=1}^i y_j G_{k-j} =  c\, T_\beta^i(y)\, \beta^{k-i} + \mathcal{O}(\alpha^{k-i}) + \mathcal{O}(1) 
\]
for all $1 \le i \le k$. 
As $0 < T_\beta^i(y) < 1$ for all $i \ge 0$, we obtain that
\[
0 \le N_k' - \sum_{j=1}^i y_j G_{k-j} < G_{k-i}
\]
for all $1 \le i \le h$, provided that $k$ is sufficiently large. 
This proves that $y_1, \ldots, y_h$ are the digits of the greedy $G$-expansion of~$N_k'$, i.e., $y_j = \epsilon_{k-j}(N_k')$ for all $1 \le j \le h$.

Finally, by \cite[Theorem~1]{akiyama2006} and \cite[Lemma~1]{frougny}, the condition \eqref{PF} implies that $\beta$ is a Pisot number.
\end{proof}

\begin{remark}
Condition~\eqref{PF} does not imply Hypothesis~B. Moreover, even \eqref{F} would not be sufficient. 
As an example let $\beta$ be ths smallest Pisot number, $\beta^3 = \beta + 1$. 
Then we have $a_1 a_2 \cdots = (10000)^\infty$, and $G$ satisfies the linear recurrence $G_k=G_{k-1}+G_{k-5}$. Its associated characteristic polynomial is $x^5-x^4-1$, which is reducible in the product $(x^3-x-1)(x^2-x+1)$. We know that \eqref{F} holds by~\cite{Akiyama}. However, Hypothesis~B does not hold because of the following relation among the elements of the recurrence.
\begin{equation*}
G_{k+3}=G_{k+1}+G_k+\begin{cases}
0 & \mbox{if}\ k\equiv 1\mod 3\\
-1 & \mbox{if}\  k\equiv -1, 0\mod 6\\
1& \mbox{if}\ k\equiv 2,3\mod 6\ .
\end{cases}
\end{equation*}
This property, easily provable by induction, shows that Hypothesis~B does not hold, since if we sum up $\tilde{N}=G_n$ and $G_{n+1}$, then in the second case considered above the first digit will change.
More generally, we have the following lemma.
\end{remark}

\begin{lemma} \label{l:BQM}
Hypothesis~B implies \eqref{QM}.
\end{lemma}

\begin{proof}
We know from Lemma~\ref{l:BPF} that $\beta$ is a Pisot number; let $\sum_{i=0}^d p_i \beta^i$ be its minimal polynomial.
Since $G_k = c\,\beta^k+\mathcal{O}(\alpha^k)$, with $c > 0$ and $0 < \alpha <\beta$, we have 
\[
f_k = \sum_{i=0}^d p_i G_{k+i} = \mathcal{O}(\alpha^k).
\]
If \eqref{QM} does not hold, then Lemma~\ref{l:QMminimal} implies that $f_k \ne 0$ for infinitely many~$k$. Assume that $f_k > 0$, the case $f_k < 0$ being symmetric. 
Hypothesis~B implies that the $G$-expansion of $\sum_{0\le i\le d:\,p_i>0} p_i G_{k+i}$ has no small terms, more precisely it ends with at least $k-\sum_{1\le i \le d:\, p_i>0} p_i b$ zeros. 
It also implies that the $G$-expansion of $\sum_{1\le i\le d:\,p_i<0} |p_i| G_{k+i} + f_k$ has small terms equal to~$f_k$ for sufficiently large~$k$. 
This contradicts that $\sum_{1\le i \le d:\,p_i>0} p_i G_{k+i} = \sum_{1\le i\le d:\,p_i<0} |p_i| G_{k+i} + f_k$. 
\end{proof}

\begin{lemma}\label{l:PF+QM-B}
The properties \eqref{PF} and \eqref{QM} imply Hypothesis~B.
\end{lemma}

\begin{proof}
By \eqref{QM} and Lemma~\ref{l:QMminimal}, arithmetic operations, that is addition and carries, on the strings of digits defining $G$-expansions can be performed in the same way as for $\beta$-expansions. 
As \eqref{PF} implies that $\beta$ is a Pisot number, there exists by \cite[Proposition~2]{frougny} some $L$ such that, for each $x \in \mathbb{Z}_+[\beta^{-1}] \cap [0,1-\beta^{-n})$ with $T_\beta^\ell(x) = 0$, $\ell \ge n$, we have $T_\beta^{\ell+L}(x+\beta^{-n}) = 0$.
This implies that addition of~$G_m$ to $N = \sum_{j=k+L}^\infty \epsilon_j(N) G_j$, $m \ge k+L$, does not change the first $k$ digits in the $G$-expansion. 
Let furthermore $L'$ be the longest run of $0$'s in the quasi-greedy $\beta$-expansion of~$1$.
Then addition of $M < G_{k-L'}$ to $N = \sum_{j=k+L}^\infty \epsilon_j(N) G_j$ or $N + G_m = \sum_{j=k}^\infty \epsilon_j(N+G_m) G_j$ is performed by concatenating the corresponding $G$-expansions.
Therefore, Hypothesis~B holds with $b = L+L'$.
\end{proof}

Lemmas~\ref{l:BPF}, \ref{l:BQM} and \ref{l:PF+QM-B} prove Theorem~\ref{t:3}.

\section{Purely discrete spectrum}
The set $\mathcal{K}_\beta$ of $\beta$-admissible sequences is
\[
\mathcal{K}_\beta = \Big\{(\epsilon_j)_{j\geq 0}\in \{0,1,\ldots, \lceil \beta \rceil-1\}^\mathbb{N}:\, \sum_{j=1}^k \frac{\epsilon_{k-j}}{\beta_j} \in [0,1)\ \mbox{for all}\ k\geq 1\Big\}
\]
and is equal to~$\mathcal{K}_G$. 
Similarly to $\delta_\infty$ in Section~\ref{sec:quot-mapp-cond}, we define $\delta'_\infty$ for $\beta = \beta_1 \in \mathbb{R}$ as
\[
\delta'_\infty:\, \mathbb{Q}(\beta) \to \mathbb{R}^{r-1} \times \mathbb{C}^s, \quad x \mapsto (x^{(2)}, \ldots, x^{(r)}, x^{(r+1)}, \ldots, x^{(r+s)}),
\]
and 
\[
\delta'_\beta:\, \mathbb{Z}[\beta] \to \mathbb{K}_\beta' = \mathbb{R}^{r-1} \times \mathbb{C}^s \times \varprojlim \mathbb{Z}[\beta] / \beta^n \mathbb{Z}[\beta], \quad x \mapsto \big(\delta'_\infty(x), \delta_\mathrm{f}(x)\big)
\]
with the natural projection $\delta_\mathrm{f}$ from $\mathbb{Z}[\beta]$ to the inverse limit $\varprojlim \mathbb{Z}[\beta] / \beta^n \mathbb{Z}[\beta]$. 
Setting 
\[
\varphi_\beta:\, \mathcal{K}_\beta \to \mathbb{K}_\beta', \quad (\epsilon_j)_{j\geq 0} \mapsto \sum_{j=0}^\infty \delta_\beta'(\epsilon_j\beta^j),
\]
the \emph{Rauzy fractal} or \emph{central tile} is defined by
\[
\mathcal{R}_\beta = \varphi_\beta(\mathcal{K}_\beta)
\]
(see e.g.\ \cite{MS}). 
Note that $\delta'_\beta$ is defined differently in~\cite{MS}; the relation to our inverse limit definition is described in~\cite{Hejda-Steiner}.

\begin{proof}[Proof of Theorem~\ref{t:1}]
For each $x\in\mathcal{K}_G$ with $\tau_G(x) \ne (0,0,\ldots)$, we have
\[
\varphi_\beta(\tau_G(x)) - \varphi_\beta(x) = \delta_\beta'(\beta^k - a_1 \beta^{k-1} - a_2 \beta^{k-2} -  \cdots - a_k) = \delta_\beta'(\tilde{T}_\beta^k(1)) 
\]
for some $k \ge 0$, thus 
\[
\varphi_\beta(\tau_G(x)) - \varphi_\beta(x) \in \delta_\beta'(V_\beta) \subset \delta_\beta'(1) + \delta_\beta'(L_\beta),
\]
with $L_\beta = \langle V_\beta - V_\beta \rangle_\mathbb{Z}$. 
By \cite[Theorem~4]{MS}, \eqref{QM} and the weak finiteness property
\begin{equation}\tag{W} \label{W}
\forall\, x \in \mathbb{Z}[\beta] \cap [0,1)\ \exists\, y \in [0,1-x), k\in\mathbb{N} \colon T_\beta^k(x+y) = T_\beta^k(y) = 0
\end{equation}
imply that $\mathcal{R}_\beta$ is a fundamental domain of $Z'/\delta_\beta'(L_\beta)$. 
Moreover, the unique invariant measure $\mu$ of $(\mathcal{K}_G, \tau_G)$ is given by $\mu=\lambda \circ \varphi_\beta$, where $\lambda$ is the Haar measure on $\mathbb{K}_\beta'$, and $\varphi_\beta$ is injective up to a set of $\mu$-measure zero.

By \cite{Barge}, every Pisot number~$\beta$ satisfies~\eqref{W}. 
Therefore, $\tau_G$ is measurably conjugate to the translation by $\delta_\beta'(1)$ on the compact group $Z'/\delta_\beta'(L_\beta)$ and has thus purely discrete spectrum.
\end{proof}

In \cite{HIT}, the spectrum of cartesian products of
odometers is investigated. 
In particular, linear recurrences $G^{1},\cdots, G^{s}$ of the multi-nacci type 
\[
G^{i}_n = a^{i} G_{n-1} + \cdots + a^{i }G_{d_i}^{i}
\]
of orders $d_{1},\ldots, d_{s}$ with with pairwise coprime positive integers $a^{i}$ $(i=1,\ldots, s)$ are considered. 
It is shown that under a certain assumption on the
independence of the dominating characteristic roots
$\beta_{1},\ldots, \beta_{s}$ of the recurrences $G^{1}, \ldots,
G^{s}$ the cartesian product of the corresponding odometers is
uniquely ergodic. The correct independence condition is
$\beta_{i}\not\in \mathbb{Q}(\beta_{j})$ (for all $i\neq j$),
whereas in \cite{HIT} a wrong condition is stated.
Note that in general such result does not hold.

When the third author gave a seminar talk in Luminy (2013), Pierre
Liardet was in the audience and gave interesting comments on the
structure of the spectrum of cartesian products of odometers. This
was the last time when the third author could meet Pierre. His
death is a great loss for mathematics as well as for his family
and all his friends. For a detailed obituary, see~\cite{memoriam}.

\section*{Acknowledgments}
The first and third author are supported by the Austrian Science Fund (FWF): Project F5510, which is a part
of the Special Research Program "Quasi-Monte Carlo Methods: Theory and Applications".  Furthermore they have received support by the Doctoral School \lq\lq Discrete Mathematics\rq\rq\ at TU Graz.
The second author is supported by the project DynA3S (ANR-13-BS02-0003) of the Agence Nationale de la Recherche. 

\bibliography{hyp_bib}

\begin{thebibliography}{10}

\bibitem{Akiyama}
S.~Akiyama.
\newblock Cubic {P}isot units with finite beta expansions.
\newblock In {\em Algebraic number theory and {D}iophantine analysis ({G}raz,
  1998)}, pages 11--26. de Gruyter, Berlin, 2000.

\bibitem{akiyama2006}
S.~Akiyama.
\newblock Positive finiteness of number systems.
\newblock In {\em Number theory}, pages 1--10. Springer, 2006.

\bibitem{Barat-Grabner16}
G.~Barat and P.~Grabner.
\newblock Combinatorial and probabilistic properties of systems of numeration.
\newblock {\em Ergodic Theory Dynam. Systems}, 36:422--457, 2016.

\bibitem{memoriam}
G.~Barat, P.~J. Grabner, and P.~Hellekalek.
\newblock Pierre {L}iardet (1943--2014) in memoriam.
\newblock {\em EMS Newsletter}, 2015.

\bibitem{Barge}
M.~Barge.
\newblock The {P}isot conjecture for $\beta$-substitutions.
\newblock {\em arXiv preprint arXiv:1505.04408}, 2015.

\bibitem{bertrand}
A.~Bertrand-Mathis.
\newblock D\'eveloppement en base {$\theta$}; r\'epartition modulo un de la
  suite {$(x\theta^n)_{n\geq 0}$}; langages cod\'es et {$\theta$}-shift.
\newblock {\em Bull. Soc. Math. France}, 114(3):271--323, 1986.

\bibitem{frougny}
C.~Frougny and B.~Solomyak.
\newblock Finite beta-expansions.
\newblock {\em Ergodic Theory and Dynamical Systems}, 12:713--723, 1992.

\bibitem{glt}
P.~Grabner, P.~Liardet, and R.~Tichy.
\newblock Odometers and systems of numeration.
\newblock {\em Acta Arithmetica}, 70:103--123, 1995.

\bibitem{Hejda-Steiner}
T.~Hejda and W.~Steiner.
\newblock Beta-expansions of rational numbers in quadratic {P}isot bases.
\newblock {\em arXiv preprint arXiv:1411.2419}, 2014.

\bibitem{HIT}
M.~Hofer, M.~R. Iac\`o, and R.~F. Tichy.
\newblock Ergodic properties of $\beta$-adic {H}alton sequences.
\newblock {\em Ergodic Theory and Dynamical Systems}, 35:895--909, 5 2015.

\bibitem{Ito_Takahashi}
S.~Ito and Y.~Takahashi.
\newblock Markov subshifts and realization of {$\beta $}-expansions.
\newblock {\em J. Math. Soc. Japan}, 26:33--55, 1974.

\bibitem{MS}
M.~Minervino and W.~Steiner.
\newblock Tilings for {P}isot beta numeration.
\newblock {\em Indag. Math. (N.S.)}, 25(4):745--773, 2014.

\bibitem{Parry}
W.~Parry.
\newblock On the {$\beta $}-expansions of real numbers.
\newblock {\em Acta Math. Acad. Sci. Hungar.}, 11:401--416, 1960.

\bibitem{Renyi}
A.~R{\'e}nyi.
\newblock Representations for real numbers and their ergodic properties.
\newblock {\em Acta Math. Acad. Sci. Hungar}, 8:477--493, 1957.

\bibitem{ST}
A.~Siegel and J.~M. Thuswaldner.
\newblock Topological properties of {R}auzy fractals.
\newblock {\em M\'em. Soc. Math. Fr. (N.S.)}, (118):140, 2009.

\end{thebibliography}
\bibliographystyle{abbrv}
\end{document}